%% file: final_arxiv.tex
\documentclass{amsart}

\usepackage[left=4cm, right=4cm]{geometry}

\usepackage{amsmath, amssymb, amsthm}
\usepackage[utf8]{inputenc}
\usepackage{url}
\usepackage{tikz}
\usepackage{stmaryrd}
\usepackage{MnSymbol}
\usepackage{paralist}
\usepackage[all]{xy}

\usepackage{hyperref}

\usepackage{comment}

\usepackage[initials]{amsrefs}

\theoremstyle{theorem}
\newtheorem{thm}{Theorem}[section]

\newtheorem{proposition}[thm]{Proposition}
\newtheorem{cor}[thm]{Corollary}

\newtheorem{lemma}[thm]{Lemma}

\theoremstyle{definition}

\newtheorem{definition}[thm]{Definition}

\newtheorem{example}[thm]{Example}
\newtheorem{rem}[thm]{Remark}

\theoremstyle{remark}

\input{Commands.tex}

\title{The logarithmic Bogomolov--Tian--Todorov theorem}

\author{Simon Felten}
\address{Institut des Hautes \'Etudes Scientifiques, 35, route de Chartres, 91440 -- Bures-sur-Yvette, France}
\email{felten@ihes.fr; felten.math@posteo.net}

\author{Andrea Petracci}
\address{Dipartimento di Matematica, Universit\`a di Bologna, Piazza di Porta San Donato 5, 40126 Bologna, Italy}
\email{a.petracci@unibo.it}

\begin{document}

\begin{abstract}
We prove that the log smooth deformations of a proper log smooth saturated log Calabi--Yau space are unobstructed.
\end{abstract}

\maketitle

\section{Introduction}

The Bogomolov--Tian--Todorov (BTT) theorem \cite{bogomolov, tian, todorov} states that deformations of cohomologically K\"ahler manifolds with trivial canonical bundle are unobstructed.
This was later proved with algebraic methods ($T^1$-lifting criterion): If $k$ is a field of characteristic $0$ and $X$ is a proper smooth $k$-variety  with torsion canonical bundle, then the deformations of $X$ are unobstructed \cite{ran, kawamata, fantechi_manetti}. This means that the functor
\[
\Def{X} : \art{k} \longrightarrow \sets
\]
which parametrizes the deformations of $X \to \Spec k$ is smooth. We refer the reader to \S\ref{sec:deformation_functors} for the definition of smoothness of a functor.
Here $\art{k}$ is the category of Artinian local $k$-algebras with residue field $k$, and $\sets$ is the category of sets.
A crucial ingredient in these proofs is that the Hodge--de Rham spectral sequence degenerates.

In this paper, we generalize the BTT theorem to the context of logarithmic schemes.
The setting is as follows.
Fix a field $k$ of characteristic $0$ and a sharp toric monoid $Q$.
Consider the monoid $k$-algebra $k[Q]$ and its completion $k \pow{Q}$ at its unique monomial maximal ideal, i.e., the ideal generated by $Q \setminus \{0\}$.
Consider the category $\art{k \pow{Q}}$ of Artinian local $k \pow{Q}$-algebras with residue field $k$.
In particular, $k \pow{Q}$ (resp.\ $k$) is the initial (resp.\ terminal) object of $\art{k \pow{Q}}$.

Every object $A$ in $\art{k \pow{Q}}$ gives rise to a log scheme $\Spec (Q \to A)$; this is the log scheme with underlying scheme $\Spec A$ and with log structure associated to the monoid homomorphism $Q \to A$ induced by the structure ring homomorphism $k \pow{Q} \to A$.

We denote by $S_0$ the log scheme $\Spec (Q \to k)$ induced by $k$. Hence $S_0$ is the log scheme with underlying scheme $\Spec k$ and with log structure $Q \oplus k^* \to k$ given by
\[
(q,a) \mapsto \begin{cases}
a \qquad \text{if } q = 0, \\
0 \qquad \text{if } q \neq 0.
\end{cases}
\]

Now we fix a log smooth saturated morphism $f_0 : X_0 \to S_0$ of log schemes.
F.~Kato has defined in \cite{Kato1996} the functor of the log smooth deformations of $f_0$, i.e., the functor
\[
\LD{X_0 / S_0} : \art{k \pow{Q}} \longrightarrow \sets
\]
which to every object $A$ in $\art{k\pow{Q}}$ associates the set of isomorphism classes of log smooth deformations $f: X \to \Spec (Q \to A)$ of $f_0: X_0 \to S_0$.

In the logarithmic setting, the Calabi--Yau condition, which appears in the BTT theorem, is expressed by triviality of the log canonical bundle $\omega_{X_0/S_0} := \Omega^d_{X_0/S_0} := \bigwedge^d\Omega^1_{X_0/S_0}$; here $d$ is the relative dimension of $f_0: X_0 \to S_0$.

Our main result is:

\begin{thm}\label{log-BTT}
Let $k$ be a field of characteristic $0$, let $Q$ be a sharp toric monoid, and let
 $S_0 = \Spec(Q \to k)$ be the log scheme with underlying scheme $\Spec k$ and ghost sheaf $Q$.
Let $f_0: X_0 \to S_0$ be a proper log smooth saturated morphism  of relative dimension $d$. If $\omega_{X_0/S_0}$ is the trivial line bundle, then the log smooth deformation functor
\[
\LD{X_0 / S_0} : \art{k \pow{Q}} \longrightarrow \sets
\]
 is smooth.
\end{thm}

The proof, which is presented in \S\ref{sec:proof}, is divided in two cases.
 If $Q = 0$, we adapt the formalism of differential graded Lie algebras, which was used by Iacono and Manetti to give an algebraic proof of the BTT theorem \cite{AbstractBTT2017, iacono_pairs, AlgebraicBTT2010}, to the context of log schemes.
 If $Q \neq 0$, we use the recent formalism developed by Chan--Leung--Ma \cite{ChanLeungMa2019} (see also \cite{chan_ma}) and by Felten--Filip--Ruddat \cite{FFR2019} to construct smoothings of degenerate Calabi--Yau varieties, and we apply an algebraic result (Proposition~\ref{prop:lifting_functors_imply_smoothness}).

\subsection{Applications}

We explain two applications of the log BTT theorem.

\subsubsection*{Log Calabi--Yau pairs}

Let $X$ be a smooth proper variety over a field $k$ of characteristic $0$ and let $D$ be an snc effective divisor on $X$.
Let $X_0$ be the log scheme given by $X$ equipped with the divisorial log structure associated to $D$.
Let $S_0$ be the log scheme given by $\Spec k$ with the trivial log structure.
Then $X_0 \to S_0$ is log smooth and saturated.
One has $\omega_{X_0/S_0} = \omega_X(D)$. Therefore Theorem~\ref{log-BTT} applies when $D$ is an anticanonical divisor, i.e., the pair $(X,D)$ is log Calabi--Yau. If this is the case, then log smooth deformations of $X_0 \to S_0$ are unobstructed.

Log smooth deformations of $X_0 \to S_0$ are exactly \emph{locally trivial} deformations of the pair $(X,D)$,
i.e., of the closed embedding $D \into X$.
The log BTT theorem implies: If $D$ is an snc effective anticanonical divisor on $X$, then the functor of locally trivial deformations of $(X,D)$ is smooth.

If $D$ is a smooth divisor, then every deformation of $(X,D)$ is locally trivial.
Therefore, if $D$ is a smooth anticanonical divisor on $X$, then deformations of the pair $(X,D)$ are unobstructed:
This recovers a result by Iacono~\cite{iacono_pairs}, Sano~\cite{sano},
and Katzarkov--Kontsevich--Pantev~\cite{KKP_Hodge_aspects}
(see also \cite{LiuRaoWan2019, Wan, Kontsevich_generalized}).

If $D$ is a non-smooth snc divisor on $D$, then there might be deformations of the pair $(X,D)$ which are not locally trivial
and hence not covered by our BTT theorem. By \cite{felten_petracci_robins}, there are---not only in characteristic $0$ but over every field $k$ with $\mathrm{char}(k) \not= 2$---indeed pairs $(X,D)$ of a smooth projective variety $X$ and an snc effective anticanonical divisor $D \subset X$ such that (not necessarily locally trivial) deformations of $(X,D)$ are obstructed.

\subsubsection*{Simple normal crossing schemes}

Fix an algebraically closed field $k$.
A \emph{normal crossing scheme} over $k$ is a scheme $X$ of finite type over $k$ such that
 every closed point $x \in X$ has an \'etale neighborhood $x \to U \to X$
which admits an \'etale map
\[U \to \Spec k[x_1, \dots,x_m]/(x_1 \cdots x_r), \]
for some $0 \leq r \leq m$, such that $x$ maps to the origin.
A normal crossing scheme $X$ is called \emph{d-semistable} if the coherent sheaf $\cExt^1_{\cO_X}(\Omega_X, \cO_X)$ is isomorphic to the structure sheaf of the singular locus of $X$ (see \cite{friedman_annals} and \cite{abramovich_et_al_log_moduli}).

If  $X$ is a d-semistable normal crossing scheme of pure dimension $d$, then $X$ has a log smooth log structure $X_0$ over $S_0 = \Spec(\NN \to k)$ such that the sheaf $\omega_{X_0/S_0}$ of log $d$-differentials is isomorphic to the dualizing sheaf $\omega_X$.
From Theorem~\ref{log-BTT} we deduce:

\begin{cor} \label{cor:d-semistable}
	Let $X$ be a d-semistable normal crossing scheme proper over an algebraically closed field of characteristic $0$ such that $\omega_X \simeq \cO_X$.
	Then the functor $\LD{X_0/S_0}$ is smooth.
\end{cor}

This removes the assumptions $H^{d-1}(X, \cO_X) = 0$ and $H^{d-2}(\tilde{X}, \cO_{\tilde{X}}) = 0$, where $d$ is the dimension of $X$ and $\tilde{X} \to X$ its normalization, from \cite[Theorem~4.2]{kawamata_namikawa}.
We remind the reader that, under the assumptions of Corollary~\ref{cor:d-semistable}, $X$ is formally smoothable by \cite{FFR2019, ChanLeungMa2019}.

\subsection{Generalizations}

In Theorem~\ref{log-BTT}, we assume $f_0: X_0 \to S_0$ log smooth. However, in many geometric situations which arise from the degeneration of varieties, the degenerate (log) space $f_0: X_0 \to S_0$ has log singularities. The deformation theory of a general log toroidal family in the sense of \cite{FFR2019} is not yet well-understood, but the special cases which arise in the Gross--Siebert program (\cite{GrossSiebertI,GrossSiebertII,Gross2011}) are. For example, according to \cite{RuddatSiebert2019}, if $f_0: X_0 \to S_0 = \Spec (\NN \to \CC)$ is a simple toric log Calabi--Yau space, then the functor $\mathcal{D}: \art{\CC\pow{t}} \to \sets$ of divisorial deformations has a hull. Moreover, the theory in \cite{FFR2019} shows that $\mathcal{D}$ satisfies Theorem~\ref{CLM-lifting} below; thus, following our proof of Theorem~\ref{log-BTT} above, the functor $\mathcal{D}$ of divisorial deformations is unobstructed. However, at least when $f_0: X_0 \to S_0$ carries a polarization, this is not a new result. Namely, $f_0: X_0 \to S_0$ is a fiber in a whole family $f: X \to \Spec (A)$ of toric log Calabi--Yau spaces over an algebraic torus $\Spec(A)$ according to \cite{GrossSiebertII}. An application of the Gross--Siebert algorithm in \cite{GHS2019}---here we use the polarization---then yields a canonical formal family $\mathfrak{X} \to \mathrm{Spf}(A\pow{t})$, where $\mathrm{Spf}(A\pow{t})$ is the completion of $\Spec (A) \subseteq \Spec(A[t])$. According to \cite{RuddatSiebert2019}, there is an analytic family $\Y \to U \times \mathbb{D}$ where $U \subseteq \Spec (A)_{an}$ is a neighborhood of the point $a \in \Spec(A)_{an}$ which corresponds to the space $f_0: X_0 \to S_0$, and $\mathbb{D}$ is a small disk; when we complete it in $(a,0) \in U \times \mathbb{D}$, the completion is isomorphic to the completion of the canonical formal family $\mathfrak{X} \to \mathrm{Spf}(A\pow{t})$ in $a \in \mathrm{Spf}(A\pow{t})$. According to \cite{RuddatSiebert2019} as well, this completion of the canonical formal family is a versal family for the divisorial deformation functor $\mathcal{D}$. Since $A$ is an algebraic torus, it is smooth, so the hull is a formally smooth $\CC\pow{t}$-algebra, and $\mathcal{D}$ is unobstructed.

\subsection*{Notation and conventions}
Every ring is commutative with identity.
The set of non-negative (resp.\ positive) integers is denoted by $\NN$ (resp.\ $\NN^+$).
Every monoid is commutative and denoted additively.
A monoid is said to be \emph{sharp} if $0$ is the unique invertible element.
A monoid $Q$ is said to be \emph{toric} if there exist an integer $n \geq 0$ and rational polyhedral cone $\sigma$ in $\RR^n$ of dimension $n$ such that $Q$ is isomorphic to $\sigma \cap \ZZ^n$.

\subsection*{Acknowledgements}
We thank Alessio Corti and Helge Ruddat for many useful conversations. The first author was funded by Helge Ruddat's DFG grant RU~1629/4-1. He thanks JGU Mainz for its hospitality.


\section{Valuations and smooth deformation functors}

\subsection{Valuations on Noetherian complete local domains}

Following \cite[VI, \S3]{bourbaki_algebre_commutative_5_7}, we define valuations on integral domains.
On the disjoint union $\NN \cup \{ \infty \}$, we consider the extensions of the sum $+$ and the ordering $\leq$ from $\NN$ defined as follows: $\infty + \infty = \infty$, $n + \infty =  \infty + n = \infty$ and $ n \leq \infty$ for every $n \in \NN$.

\begin{definition} \label{def:valuation}
	A \emph{non-trivial valuation} on a ring $\Lambda$ is a function $\nu : \Lambda \to \NN \cup \{ \infty \}$ such that:
	\begin{enumerate}
		\item $\nu (ab) = \nu(a) + \nu(b)$ for all $a,b \in \Lambda$,
		\item $\nu (a+b) \geq \min \{ \nu(a), \nu(b) \} $ for all $a,b \in \Lambda$,
		\item $\nu^{-1}(\infty) = \{ 0 \}$,
		\item $\nu (\Lambda) \setminus \{ 0, \infty \} \neq \emptyset$.
	\end{enumerate}
\end{definition}

It follows that $\Lambda$ is an integral domain and $\nu(1) = 0$.

\begin{definition}
	Let $(\Lambda, \frakm)$ be a local domain. A non-trivial valuation $\nu$ on $\Lambda$ is said to be \emph{compatible with the $\frakm$-adic topology} if the the $\frakm$-adic topology coincides with the linear topology induced by the following descending filtration of ideals:
	$
	 \{a \in \Lambda \mid \nu(a) \geq n \}
	$
	as $n \in \NN$.
\end{definition}

\begin{rem}
If a local domain $(\Lambda, \frakm)$ has a non-trivial valuation which is compatible with the $\frakm$-adic topology, then $\Lambda$ is \emph{not} a field. Whereas, in Lemma~\ref{lem:toric_rings_have_valuations} below, we construct such a valuation on $\Lambda = k\pow{Q}$ for a sharp toric monoid $Q \not= 0$, there is no such valuation on $k = k\pow{0}$.
\end{rem}

\begin{example} \label{ex:power_series_has_valuation}
Fix an integer $m \geq 1$.
Fix a field $k$ and consider the power series ring $\Lambda = k \pow{t_1, \dots, t_m}$ with maximal ideal $\frakm$. Every element $w = (w_1, \dots, w_m) \in (\NN^+)^m$ induces a non-trivial valuation $\nu_w$ on $\Lambda$ which is compatible with the $\frakm$-adic topology as follows: if \[a = \sum_{i_1, \dots, i_m \geq 0} a_{i_1, \dots, i_m} t_1^{i_1} \cdots t_m^{i_m} \]
then
\[
\nu_w(a) = \min \left\{ w_1 i_1 + \cdots + w_m i_m \mid i_1 \geq 0, \ \dots, \ i_m \geq 0, \ a_{i_1, \dots, i_m} \neq 0    \right\}.
\]
\end{example}

If $\Lambda$ is a Noetherian complete local ring, then every power series with coefficients in $\Lambda$ and finitely many variables can be evaluated on tuples of elements of the maximal ideal of $\Lambda$.
Now we give a sufficient criterion that ensures that the zero power series is the only one for which every evaluation is zero:

\begin{lemma} \label{lemma:zero_power_series}
	Let $(\Lambda, \frakm)$ be a Noetherian complete local domain, and let $f \in \Lambda \pow{x_1, \dots, x_n}$ be a power series such that
	\[
	\forall a_1 \in \frakm, \dots, \forall a_n \in \frakm, \qquad f(a_1, \dots, a_n) = 0.
	\]
	If there exists a non-trivial valuation on $\Lambda$ which is compatible with the $\frakm$-adic topology, then $f = 0$.
\end{lemma}

Note that Lemma~\ref{lemma:zero_power_series} does not hold for every Noetherian complete local ring: e.g.\ $\Lambda = k[t]/(t^{r+1})$ and $f = t^r x \in \Lambda \pow{x}$, for any integer $r \geq 0$.

In the proof of Lemma~\ref{lemma:zero_power_series} we make use of the following:

\begin{lemma}[{\cite[VII, \S3, no.~7, Lemma~2]{bourbaki_algebre_commutative_5_7}}] \label{lemma:bourbaki}
Let $\Lambda$ be a ring, and fix a non-zero power series $f(x_1, \dots, x_n) \in \Lambda \pow{x_1, \dots, x_n}$ with coefficients in $\Lambda$.
Then there exist positive integers $u_1, \dots, u_{n-1}$ such that the power series $f(x^{u_1}, \dots, x^{u_{n-1}}, x) \in \Lambda \pow{x}$ is non-zero.
\end{lemma}

\begin{proof}[Proof of Lemma~\ref{lemma:zero_power_series}]
For a contradiction assume $f \neq 0$.
By Lemma~\ref{lemma:bourbaki} there exist $u_1, \dots, u_{n-1} \in \NN^+$ such that the power series $g(x) := f(x^{u_1}, \dots, x^{u_{n-1}}, x) \in \Lambda \pow{x}$ is non-zero.
It is clear that $g(a) = f(a^{u_1}, \dots, a^{u_{n-1}}, a) = 0$ for every $a \in \frakm$.

Let $\nu : \Lambda \to \NN \cup \{ \infty \}$ be  a non-trivial valuation which is compatible with the $\frakm$-adic topology.
Set $g = \sum_{i \geq 0} b_i x^i$ with $b_i \in \Lambda$.
Let $d \geq 0$ be the minimum index $i \geq 0$ such that $b_i \neq 0$.

By (3) in Definition~\ref{def:valuation} $\nu (b_d) \in \NN$. By (4) there exists an element in $\Lambda$ whose valuation is a positive integer. By taking a sufficient high power, since $\nu$ is compatible with the $\frakm$-adic topology, we can find $a \in \frakm \setminus \{ 0 \}$ such that $\nu(a) > \nu(b_d)$.

For every $i \geq 1$ we have
\begin{align*}
\nu (b_{d+i} a^{d+i}) = \nu(b_{d+i}) + i \nu(a) + d \nu(a) \geq \nu(a) + d \nu(a) > \nu(b_d) + d \nu(a) = \nu(b_d a^d).
\end{align*}
Therefore, by (2) in Definition~\ref{def:valuation}, for every $r \geq 1$ we have
\begin{align*}
\nu \left(\sum_{i=1}^r b_{d+i} a^{d+i}\right) > \nu(b_d a^d).
\end{align*}
By taking the limit for $r \to + \infty$, since $\nu : \Lambda \to \NN \cup \{ \infty \}$ is continuous by \cite[VI, \S5, no.~1, Proposition~1]{bourbaki_algebre_commutative_5_7}, we get
\begin{equation*}
\nu \left(\sum_{j > d} b_j a^j \right) > \nu(b_d a^d).
\end{equation*}
From
\begin{equation*}
g(a) = b_d a^d + \left(\sum_{j > d} b_j a^j \right)
\end{equation*}
we deduce $\nu (g(a)) = \nu (b_d a^d) \neq \infty$, so $g(a) \neq 0$ which is a contradiction.
\end{proof}


\subsection{Formally smooth algebras} \label{sec:formally_smooth_algebras}
We fix a Noetherian complete local ring $\Lambda$; let $\m$ be its maximal ideal and let $k = \Lambda / \m$ be its residue field.

We denote by $\comp{\Lambda}$ the category of Noetherian complete local $\Lambda$-algebras with residue field $k$.
Arrows in $\comp{\Lambda}$ are $\Lambda$-algebra homomorphisms which are compatible with the projection onto $k$.
In particular, all arrows in $\comp{\Lambda}$ are local homomorphisms.
Let $\art{\Lambda}$ be the full subcategory of $\comp{\Lambda}$ consisting of Artinian rings.
It is clear that $\comp{k}$ (resp.\ $\art{k}$) is a full subcategory of $\comp{\Lambda}$ (resp.\ $\art{\Lambda}$) by considering a $k$-algebra as a $\Lambda$-algebra via $\Lambda \onto k$.

An object $R$ in $\comp{\Lambda}$ is called \emph{formally smooth (over $\Lambda$)} if it is isomorphic to the power series ring $\Lambda \pow{x_1, \dots, x_n}$ for some non-negative integer $n$. We refer the reader to \cite[Appendix~B]{talpo_vistoli} for properties of formally smooth algebras.

\begin{proposition} \label{prop:smoothness_of_hR}
	Let $(\Lambda,\frakm,k)$ be a Noetherian complete local domain with a non-trivial valuation which is compatible with the $\frakm$-adic topology.
Let
	 $R$ be an object in $\comp{\Lambda}$ such that,
for every integer $\ell \geq 1$, the function 
		\[
		\HHom_{\comp{\Lambda}} \left( R, \Lambda / \frakm^{\ell+1} \right) \longrightarrow \HHom_{\comp{\Lambda}} \left( R, \Lambda / \frakm^{\ell} \right)
		\]
		induced by $\Lambda / \frakm^{\ell +1} \onto \Lambda / \frakm^\ell$ is surjective.

	Then $R$ is a formally smooth $\Lambda$-algebra.
\end{proposition}

\begin{proof}
		Let $\frakm_R$ be the maximal ideal of $R$.
	Let $n$ be the dimension of the $k$-vector space $T_\Lambda^\vee R = \frakm_R / (\frakm R + \frakm_R^2)$.
	By \cite[Corollary~B.6]{talpo_vistoli} there exists a surjection $\Lambda \pow{x_1, \dots, x_n} \onto R$ in $\comp{\Lambda}$ whose kernel $I$ is contained in $ \frakm \pow{x_1, \dots, x_n} + (x_1, \dots, x_n)^2$.
	In other words, $R = \Lambda \pow{x_1, \dots, x_n} / I$ and every element in $I$ is of the form
	\begin{equation} \label{eq:form_elements_I}
	c_0 + c_1 x_1 + \cdots + c_n x_n + g
	\end{equation}
	where $c_0, c_1, \dots, c_n \in \frakm$ and $g$ is a power series with order $\geq 2$.

Since $\Lambda$ is $\frakm$-adically complete, for every $a = (a_1, \dots, a_n) \in \frakm \times \cdots \times \frakm$ we can consider the evaluation homomorphism 
\begin{equation*}
\mathrm{ev}_a : \Lambda \pow{x_1, \dots, x_n} \longrightarrow \Lambda
\end{equation*}
defined by $f \mapsto f(a)$.
It is a surjective local homomorphism.
 Denote by $J_a$ the image of $I$ under $\mathrm{ev}_a$; $J_a$ is an ideal of $\Lambda$ which is contained in $\frakm$.
 We want to show that $J_a = 0$.
 By contradiction let us assume that $J_a \neq 0$. By the Krull intersection theorem $\bigcap_{\ell \geq 0} \frakm^\ell = 0$. Therefore there exists an integer $\ell \geq 1$ such that $J_a \subseteq \frakm^\ell$ and $J_a \nsubseteq \frakm^{\ell + 1}$.
 Consider the composition
 \[
 \Lambda \pow{x_1, \dots, x_n} \overset{\mathrm{ev}_a}{\onto} \Lambda \onto \Lambda / J_a \onto \Lambda / \frakm^\ell;
 \]
its kernel contains $I$, therefore we have a surjective homomorphism
\[
\phiv_\ell : R = \Lambda \pow{x_1, \dots, x_n} / I \longrightarrow \Lambda / \frakm^\ell.
\]
From the assumption we deduce the existence of a homomorphism
\[
\phiv_{\ell +1} : R \longrightarrow \Lambda / \frakm^{\ell + 1}
\] such that $\phiv_\ell = \theta_{\ell} \circ \phiv_{\ell +1}$, where $\theta_{\ell} : \Lambda / \frakm^{\ell +1} \onto \Lambda / \frakm^\ell$ is the canonical projection map.
\begin{equation*}
\xymatrix{
	\\
\Lambda \pow{x_1, \dots, x_n} \ar@{->>}[r] & R = \Lambda \pow{x_1, \dots, x_n} / I \ar@{->>}[r] \ar[rrd]_{\phiv_{\ell +1}} \ar@{->>}@/^2pc/[rr]^{\phiv_\ell} & \Lambda / J_a \ar@{->>}[r] & \Lambda / \frakm^\ell \\
	& & &														\Lambda / \frakm^{\ell + 1} \ar@{->>}[u]_{\theta_{\ell}}
}
\end{equation*}
For $i=1, \dots, n$, let $b_i \in \Lambda$ such that $b_i + \frakm^{\ell +1} \in \Lambda / \frakm^{\ell +1}$ is the image of $x_i + I \in R$ via $\phiv_{\ell +1}$.
Set $b = (b_1, \dots, b_n) \in \frakm \times \cdots \times \frakm$.
In other words $\phiv_{\ell +1}$ is induced by the evaluation on $b$.
In particular this implies
\begin{equation} \label{eq:appartenenza_b}
\forall f \in I, \ f(b) \in \frakm^{\ell +1}.
\end{equation}
We have a commutative diagram
\begin{equation*}
\xymatrix{
\Lambda \pow{x_1, \dots, x_n} \ar[r]^{\ \ \ \ \ \ \mathrm{ev}_a} \ar[rd]_{\mathrm{ev}_b} & \Lambda \ar@{->>}[r] & \Lambda / \frakm^\ell \\
& \Lambda \ar@{->>}[r] & \Lambda / \frakm^{\ell + 1} \ar@{->>}[u]_{\theta_{\ell}}
}
\end{equation*}
where $\Lambda \onto \Lambda / \frakm^\ell$ and $\Lambda \onto \Lambda / \frakm^{\ell+1}$ are the canonical projection maps.
For every index $1 \leq i\leq n$, it is clear that $a_i$ and $b_i$ have the same image in $\Lambda / \frakm^\ell$, i.e., $a_i - b_i \in \frakm^\ell$. From the particular form of elements of $I$ in \eqref{eq:form_elements_I} we deduce that
\begin{equation*}
\forall f \in I, \ f(a) - f(b) \in \frakm^{\ell+1}.
\end{equation*}
By \eqref{eq:appartenenza_b} this implies $f(a) \in \frakm^{\ell +1}$ for every $f \in I$. Hence $J_a \subseteq \frakm^{\ell + 1}$, which is a contradiction. Therefore $J_a = 0$.

We have proved that
\begin{equation*}
\forall a \in \frakm \times \cdots \times \frakm, \ \forall f \in I, \ f(a) = 0.
\end{equation*}
By Lemma~\ref{lemma:zero_power_series} we deduce $I = 0$. Hence $R = \Lambda \pow{x_1, \dots, x_n}$.
\end{proof}

\subsection{Deformation functors} \label{sec:deformation_functors}
Here we briefly summarize some notions from \cite{schlessinger_functors_artin_rings}.
We fix a Noetherian complete local ring $\Lambda$ with maximal ideal $\frakm$ and residue field $k = \Lambda / \frakm$.
A \emph{deformation functor} is a functor $F : \art{\Lambda} \to \sets$ such that $F(k)$ is a singleton and $F$ satisfies Schlessinger's conditions (H1) and (H2).
If $R \in \comp{\Lambda}$ then $h_R := \sHom{\comp{\Lambda}}{R}{-}$ is a deformation functor.
If $F$ is a deformation functor, then the set $F(k[t]/(t^2))$ has a natural structure as $k$-vector space, which is called the \emph{tangent space} of $F$.

A natural transformation $F \to G$ of deformation functors is called \emph{smooth} if for every surjection $B \to A$ in $\art{\Lambda}$ the function $F(B) \to F(A) \times_{G(A)} G(B)$ is surjective.
A deformation functor $F$ is called \emph{smooth} if the natural transformation $F \to h_\Lambda$ is smooth.
For $R \in \comp{\Lambda}$, $R$ is a formally smooth $\Lambda$-algebra if and only if $h_R$ is smooth.

A \emph{hull} for a deformation functor $F$ is an object $R \in \comp{\Lambda}$ such that there exists a smooth natural transformation $h_R \to F$ which induces a bijection on tangent spaces. If a hull exists, it is unique.
	
Now we give a sufficient criterion for the smoothness of a deformation functor.

\begin{proposition} \label{prop:lifting_functors_imply_smoothness}
	Let $(\Lambda,\frakm,k)$ be a Noetherian complete local domain with a non-trivial valuation which is compatible with the $\frakm$-adic topology.
	Let $F : \art{\Lambda} \to \sets$ be a deformation functor with finite dimensional tangent space.

If,	for every integer $\ell \geq 1$, the function 
	\[
	F \left( \Lambda / \frakm^{\ell+1} \right) \longrightarrow F \left( \Lambda / \frakm^{\ell} \right)
	\]
	induced by $\Lambda / \frakm^{\ell +1} \onto \Lambda / \frakm^\ell$ is surjective, then $F$ is smooth.
\end{proposition}

\begin{proof}
By \cite[Theorem~2.11]{schlessinger_functors_artin_rings} the functor $F$ has a hull $R$. Consider a smooth map $h_R \to F$.

Fix an arbitrary integer $\ell \geq 1$. 
Since
$
F(\Lambda / \frakm^{\ell + 1}) \to F(\Lambda / \frakm^\ell)
$
is surjective,
\begin{equation} \label{eq:smoothness_x}
F(\Lambda / \frakm^{\ell + 1}) \times_{F(\Lambda / \frakm^\ell)} h_R (\Lambda / \frakm^\ell) \longrightarrow h_R(\Lambda / \frakm^{\ell })
\end{equation}
is surjective.
Since $h_R \to F$ is smooth, the function
\begin{equation} \label{eq:smoothness_y}
h_R (\Lambda / \frakm^{\ell +1}) \longrightarrow F(\Lambda / \frakm^{\ell + 1}) \times_{F(\Lambda / \frakm^\ell)} h_R (\Lambda / \frakm^\ell)
\end{equation}
is surjective.
By composing \eqref{eq:smoothness_y} and \eqref{eq:smoothness_x}, we obtain that $h_R(\Lambda / \frakm^{\ell + 1}) \to h_R(\Lambda / \frakm^\ell)$ is surjective.
By Proposition~\ref{prop:smoothness_of_hR}, $h_R$ is smooth.

We conclude with \cite[Proposition~2.5(iii)]{schlessinger_functors_artin_rings}.
\end{proof}

Recall from \S\ref{sec:formally_smooth_algebras} that $\art{k}$ is a full subcategory of $\art{\Lambda}$.

\begin{lemma}\label{restricted-hull}
	Let $F: \art{\Lambda} \to \sets$ be a deformation functor. Consider the restricted functor $F_0 := F \vert_{\art{k}}: \art{k} \to \sets$. Then:
	\begin{enumerate}
		\item $F_0$ is a deformation functor;
		\item if $R \in \comp{\Lambda}$ is the hull of $F$, then $R \otimes_\Lambda k$ is the hull of $F_0$.
	\end{enumerate}
\end{lemma}

\begin{proof}
	(1) Since the embedding $\art{k} \into \art{\Lambda}$ preserves pull-backs of Artinian rings, the restricted functor $F_0$ satisfies Schlessinger's conditions (H1) and (H2).
	
	(2) It is clear that the tangent space of $F_0$ coincides with the tangent space of $F$.
	Consider a smooth map $h_R \to F$ which induces a bijection on tangent spaces.
 It remains smooth after restriction to $\art{k}$, so we have a smooth map $(h_R)_0 \to F_0$ of deformation functors which induces a bijection on tangent spaces. The surjection $R \to R \otimes_\Lambda k$ induces an isomorphism
	 \[
	 h_{R \otimes_\Lambda k} = \mathrm{Hom}_k(R \otimes_\Lambda k,-) \to \mathrm{Hom}_\Lambda(R,-) = (h_R)_0
	 \]
	of functors.
\end{proof}


\section{Proofs}

\subsection{Toric rings} \label{sec:toric}

Fix a field $k$ and monoid $Q$.
Consider the monoid $k$-algebra $k[Q]$ defined as follows: the underlying $k$-vector space of $k[Q]$ has a basis $\{z^q\}_{q \in Q}$ indexed by elements of $Q$ and the product on $k[Q]$ is the $k$-linear extension of the rule $z^{q_1} \cdot z^{q_2} = z^{q_1 + q_2}$ for all $q_1, q_2 \in Q$.
If $Q$ is finitely generated, then $k[Q]$ is of finite type over $k$.

If $Q$ is sharp, then the ideal $(z^q \mid q \in Q \setminus \{0\})$ is a maximal ideal in $k[Q]$; we denote by $k \pow{Q}$ the completion of $k[Q]$ with respect to this maximal ideal. Thus, if $Q$ is finitely generated and sharp, then $k \pow{Q}$ is a Noetherian complete local ring with residue field $k$.

\begin{lemma} \label{lem:toric_rings_have_valuations}
	If $k$ is a field and $Q$ is a non-zero sharp toric monoid, then $k \pow{Q}$ admits a non-trivial valuation which is compatible with the adic topology of the maximal ideal of $k \pow{Q}$.
\end{lemma}

\begin{proof}
This is a generalization of Example~\ref{ex:power_series_has_valuation}.
Set $M \simeq \ZZ^n$ and let $\sigma$ be an $n$-dimensional rational polyhedral cone in $M_\RR$ such that $Q = \sigma \cap M$.
We consider the dual lattice $N = \HHom_\ZZ(M,\ZZ)$ with duality pairing $\langle - , - \rangle : M \times N \to \ZZ$ and  the dual cone
\[
\sigma^\vee = \{ v \in N_\RR \mid \forall q \in \sigma, \langle q, v \rangle \geq 0 \}.
\]
Pick $w \in N$ in the interior of $\sigma^\vee$, and consider the valuation $\nu_w : k \pow{Q} \to \NN \cup \{ \infty \}$ defined by
\[
\nu_w \left( \sum_{q \in Q} b_q z^q \right) = \min \{ \langle q, w \rangle \mid q \in Q, b_q \neq 0 \}.
\]
This is compatible with the adic topology of the maximal ideal of $k \pow{Q}$.
\end{proof}

\subsection{Log smooth deformation theory}

We fix a field $k$ (of arbitrary characteristic) and a sharp toric monoid $Q$.
As explained in the introduction, every ring $A$ in $\art{k \pow{Q}}$ gives rise to a log scheme $\Spec(Q \to A)$ with underlying scheme $\Spec A$.
In particular, $k$ gives rise to $S_0 := \Spec (Q \to k)$, which is the log scheme on the scheme $\Spec k$ with ghost sheaf $Q$.

We fix a proper log smooth saturated morphism $f_0: X_0 \to S_0$ of log schemes of relative dimension $d$. Let $\Omega^1_{X_0/S_0}$ be the sheaf of log differentials of $X_0$ relative to $S_0$; let furthermore $\omega_{X_0/S_0} = \Omega^d_{X_0/S_0}$ be the log canonical line bundle on $X_0$.

F.~Kato defines in \cite{Kato1996} the functor of log smooth deformations of $X_0 \to S_0$, i.e., the functor
\[
\LD{X_0 / S_0} : \art{k \pow{Q}} \longrightarrow \sets
\]
which to every object $A$ in $\art{k \pow{Q}}$ associates the set of isomorphism classes of log smooth deformations $f: X \to \Spec (Q \to A)$ of $f_0: X_0 \to S_0$ (see \cite[Definition~8.1]{Kato1996}).
This functor is a deformation functor and has a hull by \cite[Theorem~8.7]{Kato1996}.

Now we assume furthermore that $k$ has characteristic $0$. In \cite{Felten2020}, the first author proves that $\LD{X_0/S_0}$ is controlled by a $k \pow{Q}$-linear \emph{pre}differential graded Lie algebra (pdgla, for short) $(L^\bullet_{X_0/S_0},[-,-],d,\ell)$.
More precisely, $(L^\bullet_{X_0/S_0},[-,-])$ is a graded Lie algebra over $k \pow{Q}$ endowed with a derivation $d$ which need not be a differential but admits an element $\ell \in L^2_{X_0/S_0}$ such that $d^2 = [\ell,-]$.
Via a modified Maurer--Cartan equation we associate a deformation functor $\art{k \pow{Q}} \to \sets$, which is isomorphic to $\LD{X_0 / S_0}$. 
Directly from the definitions it follows that the restricted functor $\LD{X_0/S_0} \vert_{\art{k}}$ is controlled by the $k$-linear (ordinary) dgla $L^\bullet_{X_0/S_0} \otimes_{k\pow{Q}} k$. In particular, it has a hull; Lemma~\ref{restricted-hull} shows that the hull of $\LD{X_0/S_0} \vert_{\art{k}}$ is $R \otimes_{k\pow{Q}} k$ where $R$ is the hull of $\LD{X_0/S_0}$.

\subsection{Homotopy abelianity} \label{sec:homotopy_abelian}

In this section, we prove that the restricted deformation functor $\LD{X_0/S_0} \vert_{\art{k}}$ is unobstructed; in particular, we obtain the log BTT theorem in the case $Q = 0$. We fix a field $k$ of characteristic $0$. 

Recall that a differential graded Lie algebra (dgla, for short) $L^\bullet$ is \emph{abelian} if $[-,-] = 0$, and \emph{homotopy abelian} if it is quasi-isomorphic to an abelian dgla. This is the case if and only if $L^\bullet$ is formal and $H^\bullet(L^\bullet)$ is abelian; in this case, the associated deformation functor $\Def{L^\bullet} : \art{k} \to \sets$ is smooth. We say that a deformation functor $F : \art{k} \to \sets$ is \emph{controlled} by the dgla $L^\bullet$ if it is isomorphic to $\Def{L^\bullet}$. We refer the reader to \cite{manetti_seattle, cosimplicial_dglas_in_def_theory} for details.

Now we state Iacono's abstract Bogomolov--Tian--Todorov theorem from \cite{AbstractBTT2017} because it will be useful to us below.
Recall that a \emph{Cartan homotopy} $\mathbf{i}: L^\bullet \to M^\bullet$ of dglas is a homogeneous linear map of degree $-1$ such that
\[
\mathbf{i}_{[a,b]} = [\mathbf{i}_a,d_M\mathbf{i}_b] \quad \mathrm{and} \quad [\mathbf{i}_a,\mathbf{i}_b] = 0
\]
are satisfied; here, $\mathbf{i}_a := \mathbf{i}(a) \in M^\bullet$ is traditional notation since the elements of $M^\bullet$ are typically homomorphisms themselves---compare this with our choice for $M^\bullet$ below. The \emph{Lie derivative} of $\mathbf{i}$ is given by $\mathbf{l}_a = d_M\mathbf{i}_a + \mathbf{i}_{d_La}$; it defines a homomorphism $\mathbf{l}: L^\bullet \to M^\bullet$ of complexes, which is homotopic to $0$ via the homotopy $\mathbf{i}$.

\begin{thm}[{Abstract Bogomolov--Tian--Todorov theorem \cite[Theorem~3.3]{AbstractBTT2017}}] \label{aBTT}
	Let $L^\bullet, M^\bullet$ be dglas over a field $k$ of 
	characteristic $0$, let $\mathbf{i}: L^\bullet \to M^\bullet$ be a Cartan homotopy, let $H^\bullet \subseteq M^\bullet$ 
	be a sub-dgla, and assume that 
	\begin{enumerate}\itemsep 0pt
		\item $\mathbf{l}_a \in H^\bullet$ for every $a \in L^\bullet$;
		\item $H^\bullet \to M^\bullet$ is injective in cohomology, i.e., $H^\bullet(H^\bullet) \to H^\bullet(M^\bullet)$ is injective;
		\item the morphism $\mathbf{i}: L^\bullet \to M^\bullet/H^\bullet[-1]$ of complexes is injective in cohomology.
	\end{enumerate}
	Then $L^\bullet$ is homotopy abelian.
\end{thm}

Now we move to the logarithmic setting. Recall that the restricted deformation functor $\LD{X_0/S_0} \vert_{\art{k}}$ is controlled by the $k$-linear dgla $L^\bullet_{X_0/S_0} \otimes_{k \pow{Q}} k$. We prove:

\begin{thm}\label{hom-ab-naive}
	In the setting of Theorem~\ref{log-BTT}, if $\omega_{X_0/S_0}$ is the trivial line bundle, then the $k$-linear dgla $L^\bullet_{X_0/S_0} \otimes_{k \pow{Q}} k$ is homotopy abelian. In particular, the restricted functor $\LD{X_0/S_0} \vert_{\art{k}}$ is smooth.
\end{thm}

In the remainder of this section, we explain the proof of Theorem~\ref{hom-ab-naive}.

The construction of the $k \pow{Q}$-pdgla $L^\bullet_{X_0/S_0}$ in \cite{Felten2020} relies on the Thom--Whitney resolution; we briefly recall its basic properties. For a more complete treatment, cf.~\cites{AznarHodgeDeligne1987,AlgebraicBTT2010,Felten2020}. Given a sheaf $\F$ of $k$-vector spaces on $X_0$ and an affine Zariski open cover $\U = \{U_i\}$ of $X_0$, we obtain the semicosimplicial \v Cech resolution $\F(\U)$, which is a semicosimplicial sheaf on $X_0$. Similarly, when $\G^\bullet$ is a complex of sheaves, we obtain a semicosimplicial complex of sheaves $\G^\bullet(\U)$. Now the Thom--Whitney resolution $\TW^{\bullet,\bullet}(\G^\bullet(\U))$ is a double complex---in the sense that $d_1d_2 + d_2d_1 = 0$---such that $\G^p \to \TW^{p,\bullet}(\G^\bullet(\U))$ is a resolution of the sheaf $\G^p$. If $\G^p$ is quasi-coherent, then the resolution is acyclic so that it computes the cohomology $H^q(X_0,\G^p)$. The Thom--Whitney resolution is more subtle than the usual \v{C}ech complex $\check \C^\bullet(\U;\G^\bullet)$, which can be obtained by just taking alternating sums of boundary maps in $\G^\bullet(\U)$; this additional complexity allows to extend algebraic structures (like a Lie bracket or a product) from $\G^\bullet$ to the actual complex $\TW^{\bullet,\bullet}(\G^\bullet(\U))$, not only to its cohomology $H^\bullet(X_0,\G^\bullet)$---cf.~the construction of the cup product in cohomology. In the notation $\TW^{p,q}(-)$ of \cite{Felten2020}, there is a switch of indices relative to the standard notation---i.e., $\TW^{p,q}(-) := C_\TW^{q,p}(-)$. This does not affect the formation of the total complex $\mathrm{Tot}_\TW(-)$. The Thom--Whitney resolution does not naively commute with the shift functor; instead, we have an isomorphism
\[
s: \mathrm{Tot}_\TW(\G^\bullet(\U))[m] \xrightarrow{\cong} \mathrm{Tot}_\TW(\G^\bullet[m](\U))
\]
which is given on $\TW^{p,q}(-)$ by multiplication with $(-1)^{qm}$. Finally, note that our construction of the Thom--Whitney resolution follows the standard but differs from \cite{AlgebraicBTT2010} by the order of the tensor factors; this makes a difference in some signs.

In \cite{Felten2020}, the Thom--Whitney resolution $\mathrm{PV}_{X_0/S_0} := \TW^{\bullet,\bullet}(G^\bullet_{X_0/S_0}(\U))$ of the Gerstenhaber algebra $G^\bullet_{X_0/S_0}$ of polyvector fields is considered; it gives rise to a Thom--Whitney resolution $G^{-1}_{X_0/S_0} \to \mathrm{PV}^{-1,\bullet}_{X_0/S_0}$ of log derivations. It is clear from the construction in \cite{Felten2020} that 
\[
L^\bullet_{X_0/S_0} \otimes_{k \pow{Q}} k \cong \Gamma(X_0,\mathrm{PV}^{-1,\bullet}_{X_0/S_0}) = \Gamma(X_0,\mathrm{Tot}_\TW(\Theta^1_{X_0/S_0}(\U))),
\]
where $\Theta^1_{X_0/S_0}$ is considered as a complex concentrated in degree $0$.

\begin{proof}[{Proof of Theorem~\ref{hom-ab-naive}}]
	In preparation for applying Theorem~\ref{aBTT} in our situation, we apply the semicosimplicial \v Cech resolution to the de Rham complex $\Omega^\bullet_{X_0/S_0}$; this yields the semicosimplicial \v Cech complex $\Omega^\bullet_{X_0/S_0}(\U)$. We denote the singly graded total complex of the Thom--Whitney resolution by $\mathrm{T}\Omega^\bullet := \Gamma(X_0, \mathrm{Tot}_\TW(\Omega^\bullet_{X_0/S_0}(\U)))$. Similarly, we write $\mathrm{T}\Omega^i := \Gamma(X_0,\mathrm{Tot}_\TW(\Omega^i_{X_0/S_0}[-i](\U)))$ for the resolutions of the individual sheaves of differential forms. Then we have injective graded maps $\mathrm{T}\Omega^i \to \mathrm{T}\Omega^\bullet$ because the Thom--Whitney construction is exact. Namely, $\mathrm{Tot}_\mathrm{TW}(-(\U))$ transforms exact sequences of sheaves into exact sequences of sheaves by the construction in \cite[5.3]{Felten2020} and the exactness result in \cite[2.4]{AznarHodgeDeligne1987}; then $\Gamma(X_0,-)$ is only left exact, but $\mathrm{Tot}_\mathrm{TW}(\F(\U))$ is acyclic at least whenever $\F$ is quasi-coherent by \cite[5.6]{Felten2020}.
	However, only for $i = d$, the graded linear map $\mathrm{T}\Omega^d \to \mathrm{T}\Omega^\bullet$ is a homomorphism of complexes since only $\Omega^d_{X_0/S_0}[-d] \to \Omega^\bullet_{X_0/S_0}$ is compatible with the differential. Here, $d$ is the relative dimension of $f_0: X_0 \to S_0$. Moreover, the composed map $\mathrm{T}\Omega^{d-1} \to \mathrm{T}\Omega^\bullet/\mathrm{T}\Omega^d$ is a homomorphism of complexes since $\mathrm{T}\Omega^\bullet/\mathrm{T}\Omega^d$ is the resolution of $\Omega^\bullet_{X_0/S_0}/\Omega^d_{X_0/S_0}[-d]$ and $\Omega^{d-1}_{X_0/S_0} \to \Omega^\bullet_{X_0/S_0}/\Omega^d_{X_0/S_0}[-d]$ is a homomorphism of complexes. The two homomorphisms of complexes are injective in cohomology by the argument in \cite[Thm.~6.4]{AlgebraicBTT2010} and the fact that the Hodge--de Rham spectral sequence $H^q(X_0,\Omega^p_{X_0/S_0}) \Rightarrow \HH^{p + q}(X_0, \Omega^\bullet_{X_0/S_0})$ degenerates at $E_1$---this is proven in \cite{Illusie2005} as well as a special case of \cite[Thm.~1.9]{FFR2019}. 
	
	We now set $L^\bullet := L^\bullet_{X_0/S_0} \otimes_{k \pow{Q}} k$ and $M^\bullet := \mathrm{Hom}_k(\mathrm{T}\Omega^\bullet,\mathrm{T}\Omega^\bullet)$; the latter space is a dgla by the construction in \cite[Ex.~2.2]{AbstractBTT2017}. Using the K\"unneth formula, we compute its cohomology via the natural isomorphism
	\[
	H^\bullet(M^\bullet) \xrightarrow{\cong} \mathrm{Hom}_k^\bullet(H^\bullet(\mathrm{T}\Omega^\bullet), H^\bullet(\mathrm{T}\Omega^\bullet)).
	\]
	A \emph{contraction} $\invneg\ : L^\bullet \times \mathrm{T}\Omega^\bullet \to \mathrm{T}\Omega^\bullet$ is (by definition) a bilinear map of degree $-1$ such that the induced map $L^\bullet \to M^\bullet = \mathrm{Hom}_k(\mathrm{T}\Omega^\bullet,\mathrm{T}\Omega^\bullet)$ is a Cartan homotopy. We construct such a contraction---and thus a Cartan homotopy---from a \emph{semicosimplicial contraction}, i.e., a system of contractions $\invneg_n : \Theta^1_{X_0/S_0}(\U)_n \times \Omega^\bullet_{X_0/S_0}(\U)_n \to \Omega^\bullet_{X_0/S_0}(\U)_n$ which is compatible with the coface maps of the semicosimplicial complex---cf.~\cite{AbstractBTT2017}. Concretely, this contraction is given by contracting log derivations with differential forms---i.e., by the unique map $\invneg\ : \Theta^1_{X_0/S_0} \times \Omega^p_{X_0/S_0} \to \Omega^{p - 1}_{X_0/S_0}$ which satisfies $\theta \ \invneg\ \omega = \langle\omega,\theta\rangle$ for 
	$\omega \in \Omega^1_{X/S}$ and
	\[
	\theta \ \invneg \ (\omega \wedge \eta) = (\theta \ \invneg\ \omega) \wedge \eta 
	+ (-1)^{|\omega|}\omega \wedge (\theta \ \invneg\ \eta).
	\]
	We set 
	\[
	H^\bullet := \{m \in M^\bullet \ | \ m(\mathrm{T}\Omega^d) \subset \mathrm{T}\Omega^d\} \subset M^\bullet;
	\]
	it is a sub-dgla whose embedding is injective in cohomology by \cite[Ex.~2.17]{AbstractBTT2017} because $\mathrm{T}\Omega^d \subset \mathrm{T}\Omega^\bullet$ is a sub-dg-vector space whose embedding is injective in cohomology.
	
	We show the condition $\mathbf{l}_\theta \in H$ by explicit computation. We have decompositions
	\[
	L^\bullet = \bigoplus_\lambda \Gamma(X_0,\TW^{0,\lambda}(\Theta^1_{X_0/S_0}(\U)))
	\quad \text{and} \quad 
	\mathrm{T}\Omega^d = \bigoplus_\mu \Gamma(X_0,\TW^{d,\mu - d}(\Omega^\bullet_{X_0/S_0}(\U))),
	\]
	so let $\theta = (t_n \otimes \theta_n)_n \in L^\bullet$ be of bidegree $(0,\lambda)$ 
	and $\omega = (a_n \otimes \omega_n)_n \in \mathrm{T}\Omega^d$ be of bidegree $(d,\mu-d)$. 
	Then
	\begin{align}
	\mathbf{l}_\theta(\omega) &= d\mathbf{i}_\theta(\omega) 
	+ (-1)^\lambda \mathbf{i}_\theta(d\omega) + \mathbf{i}_{d\theta}(\omega) \nonumber \\
	&= \big((t_n \wedge a_n) \otimes d(\theta_n \ \invneg \ \omega_n)\big)_n 
	\in \Gamma(X_0,\TW^{d,\mu - d + \lambda}(\Omega^\bullet_{X_0/S_0}(\U))).\nonumber
	\end{align}
	
	To prove the final condition, let $\eta \in \Omega^d_{X_0/S_0}$ be a volume form. In the diagram 
	\[
	\xymatrix{
		L^\bullet \ar[r]^-{\alpha} \ar[d]^-\gamma & \mathrm{Hom}_k^\bullet(\mathrm{T}\Omega^d, \mathrm{T}\Omega^{d - 1})[-1] 
		\ar[d]^{\mathrm{ev}_\eta} \\
		\mathrm{Tot}_\mathrm{TW}(\Omega^{d - 1}_{X_0/S_0}(\U)) & \mathrm{T}\Omega^{d - 1}[d - 1] \ar[l]_-s 
	}
	\]
	the map $\alpha$ induced by contraction is a homomorphism of complexes:
	In a computation analogous 
	to that of $\mathbf{l}_\theta(\omega)$, the differentials $d(\theta_n \ \invneg \ \omega_n)$ vanish, for $\Omega^{d - 1}_{X_0/S_0}[-d + 1]$ is concentrated in a single degree. The 
	evaluation $\mathrm{ev}_\Omega$ maps a morphism $\phi: \mathrm{T}\Omega^d \to \mathrm{T}\Omega^{d - 1}$ to 
	the image $(-1)^{d|\phi|}\phi(1 \otimes \Omega)$ where 
	$1 \otimes \Omega \in \mathrm{T}\Omega^d$ is the element induced by $\Omega$ in degree $d$ and 
	$|\phi|$ is the 
	degree of the original map (without shift in the Hom complex).
	It is a homomorphism of complexes because $d\Omega = 0$. The map $s$ 
	is the above mentioned comparison map for the two shifted versions of the Thom--Whitney resolution. When we evaluate the contraction at $\Omega$, this gives an isomorphism $\Theta^1_{X_0/S_0} \cong \Omega^{d - 1}_{X_0/S_0}$, and $\gamma$ is the induced 
	map of Thom--Whitney complexes. The diagram is commutative, so $\alpha$ is 
	injective in cohomology (as $\gamma$ is an isomorphism). The homomorphism $\alpha$ fits into a diagram
	\[
	\xymatrix{
		L^\bullet \ar[r] \ar[d]_-\alpha & M^\bullet/H^\bullet[-1] \ar[d] \\
		\mathrm{Hom}_k^\bullet(\mathrm{T}\Omega^d, \mathrm{T}\Omega^{d-1})[-1] \ar[r]^-\eta & \mathrm{Hom}_k^\bullet(\mathrm{T}\Omega^d, \mathrm{T}\Omega^\bullet/\mathrm{T}\Omega^d)[-1]; \\
	}
	\]
	because $\mathrm{T}\Omega^{d - 1} \to \mathrm{T} \Omega^\bullet/\mathrm{T}\Omega^d$ is injective in cohomology, $\eta$ is 
	injective in cohomology by the K\"unneth theorem. Thus, $L^\bullet \to M^\bullet/H^\bullet[-1]$ is injective in cohomology.
\end{proof}

\begin{rem}
	Theorem \ref{hom-ab-naive} remains true if we relax the Calabi--Yau condition in the sense that we only require that $\omega_{X_0/S_0}^{\otimes N}$ is trivial  for some $N > 0$. In this case, the $N$-cyclic covering
	\[
	\pi: Y_0 := X_0[\!\!\sqrt[N]{\omega_{X_0/S_0}}] \to X_0
	\]
    is a finite \'etale covering; we endow $Y_0$ with the induced log structure from $X_0$, thus we turn $Y_0 \to S_0$ into a proper and saturated log smooth morphism with $\omega_{Y_0/S_0} \cong \cO_{Y_0}$. The canonical map $\Theta^1_{X_0/S_0} \to \pi_*\Theta^1_{Y_0/S_0}$ induces a homomorphism
	\[
	L_{X_0/S_0}^\bullet \otimes_{k \pow{Q}} k \to L^\bullet_{Y_0/S_0} \otimes_{k\pow{Q}} k
	\]
	of dglas, which is injective in cohomology (because $H^\bullet(X_0,\Theta^1_{X_0/S_0}) \to H^\bullet(Y_0,\Theta^1_{Y_0/S_0})$ is injective). Thus, $L_{X_0/S_0}^\bullet \otimes_{k \pow{Q}} k$ is homotopy abelian by \cite[2.11]{AbstractBTT2017}.
	
	We do not know if Theorem~\ref{log-BTT} remains true in this situation because it is unclear if Theorem~\ref{CLM-lifting} below holds. The latter relies on the existence of a volume form, which is used to transport the de Rham differential to the polyvector fields and thus construct the Batalin--Vilkovisky operator; such a volume form is not given in the situation where the log canonical bundle is only a torsion line bundle but not trivial. 
\end{rem}
\begin{rem}
 In case $Q \not= 0$, the smoothness of $\LD{X_0/S_0} \vert_{\art{k}}$ follows from Theorem~\ref{CLM-lifting} below as well. This second proof does not generalize to the torsion case discussed in the Remark above.
\end{rem}
\begin{rem}
 Theorem~\ref{hom-ab-naive} holds as well for log toroidal families $f_0: X_0 \to S_0$ in the sense of \cite{FFR2019}; we do not need to construct a version of $L^\bullet_{X_0/S_0}$ in this case---the correct dgla is obtained by the Thom--Whitney resolution of the reflexive sheaf $\Theta^1_{X_0/S_0}$. This dgla then controls the functor of locally trivial log deformations---a notion which makes sense on the subcategory $\art{k} \subseteq \art{k\pow{Q}}$ but not on the full category $\art{k\pow{Q}}$. Theorem~\ref{hom-ab-naive} holds because the Hodge--de Rham spectral sequence of the reflexive de Rham complex $W^\bullet_{X_0/S_0}$ degenerates at $E_1$. As in the log smooth case, it suffices to assume that $\omega_{X_0/S_0}$ is a torsion line bundle; in fact, the cyclic covering $X_0[\!\!\sqrt[N]{\omega_{X_0/S_0}}] \to S_0$  of a log toroidal family is a log toroidal family as well. 
\end{rem}


\subsection{Proof of Theorem~\ref{log-BTT}}
 \label{sec:proof}
 
 We start by recalling a fundamental result about smoothings of log Calabi--Yau spaces:
 
 \begin{thm}[{Chan--Leung--Ma~\cite{ChanLeungMa2019}, Felten--Filip--Ruddat~\cite{FFR2019}, Felten~\cite{Felten2020}}] \label{CLM-lifting}
 	Let $k$, $Q$, and $X_0 \to S_0$ be as in Theorem~\ref{log-BTT}. 
 	Denote by $\frakm$ the maximal ideal of $k \pow{Q}$.
 	If $\omega_{X_0/S_0}$ is the trivial line bundle, then the function
 	\[
 	\LD{X_0 / S_0} (k \pow{Q} / \frakm^{\ell + 1}) \longrightarrow \LD{X_0 / S_0} (k \pow{Q} / \frakm^{\ell})
 	\] 
 	is surjective for every integer $\ell \geq 1$.
 \end{thm}

 \begin{proof}
 	This follows from an algebraic version of the method in \cite{ChanLeungMa2019}, which is also employed in \cite{FFR2019}.
 	Given a Maurer--Cartan solution in $L_{X_0/S_0}^\bullet \otimes_{k \pow{Q}} \left(k \pow{Q} / \frakm^\ell\right)$, we must find a lifting in $L_{X_0/S_0}^\bullet \otimes_{k \pow{Q}}\left( k \pow{Q} / \frakm^{\ell + 1}\right)$.
 	Its existence follows from the analogue of \cite[Thm.~5.5]{ChanLeungMa2019} once we algebraize the theory in \cite{ChanLeungMa2019} in the spirit of \cite{Felten2020}.
 	The crucial ingredient here is the fact that the Hodge--de Rham spectral sequence
 	\[
 	H^q(X_0,\Omega^p_{X_0/S_0}) \Rightarrow \HH^{p + q}(X_0, \Omega^\bullet_{X_0/S_0})
 	\]
 degenerates at $E_1$---see \cite{Illusie2005} or \cite{FFR2019}---, and that $\omega_{X_0/S_0} \cong \cO_{X_0}$. The algebraization is rather straightforward; the first author will elaborate it in a separate paper, where he provides the technical machinery in more generality.
 \end{proof}

\begin{proof}[{Proof of Theorem~\ref{log-BTT}}]
If $Q = 0$, we use Theorem~\ref{hom-ab-naive}.
If $Q \neq 0$, we combine Theorem~\ref{CLM-lifting}, Lemma~\ref{lem:toric_rings_have_valuations} and Proposition~\ref{prop:lifting_functors_imply_smoothness}.
\end{proof}

\bibliography{log-BTT.bib}

\end{document}

%% file: Commands.tex
\newcommand{\art}[1]{\mathbf{Art}_{#1}}
\newcommand{\comp}[1]{\mathbf{Comp}_{#1}}
\newcommand{\sets}{\mathbf{Set}}
\newcommand{\Def}[1]{\mathrm{Def}_{#1}}
\newcommand{\LD}[1]{\mathrm{LD}_{#1}}

\newcommand{\phiv}{\varphi}

\newcommand{\TW}{\mathrm{TW}}

\newcommand{\NN}{\mathbb{N}}
\newcommand{\ZZ}{\mathbb{Z}}
\newcommand{\CC}{\mathbb{C}}

\newcommand{\RR}{\mathbb{R}}

\newcommand{\HH}{\mathbb{H}}

\newcommand{\U}{\mathcal{U}}

\newcommand{\F}{\mathcal{F}}
\newcommand{\G}{\mathcal{G}}

\newcommand{\cO}{\mathcal{O}}

\newcommand{\C}{\mathcal{C}}

\newcommand{\Y}{\mathcal{Y}}

\newcommand{\m}{\mathfrak{m}}
\DeclareMathOperator{\Spec}{Spec}

\newcommand{\sHom}[3]{\mathrm{Hom}_{#1}(#2 , #3)}

\newcommand{\pow}[1]{\llbracket  #1 \rrbracket}

\renewcommand{\setminus}{\smallsetminus}

\newcommand{\into}{\hookrightarrow} 
\newcommand{\onto}{\twoheadrightarrow} 

\newcommand{\cExt}{\mathcal{E}xt}

\newcommand{\HHom}{\mathrm{Hom}} 

\newcommand{\frakm}{\mathfrak{m}}